\documentclass[12pt,a4paper,pdfps]{article}
\usepackage[cp1251]{inputenc}
\usepackage[english]{babel}
\usepackage{color}
\usepackage[colorlinks,linkcolor=blue,filecolor=blue,citecolor=blue]{hyperref}
\usepackage{cite}
\usepackage{amsmath,amssymb,amsfonts,amsthm}
\usepackage{algorithmic}
\usepackage{graphicx}
\usepackage{textcomp}
\usepackage{xcolor}
\usepackage{subfig}
\usepackage[left=30mm, top=20mm, right=20mm, bottom=20mm, nohead, footskip=7mm]{geometry}
\usepackage{setspace}
\newtheorem{theorem}{Theorem}
\newtheorem{proposition}{Proposition}

\DeclareMathOperator{\Lip}{Lip}

\DeclareMathOperator{\card}{card}
\DeclareMathOperator{\ad}{ad}
\newcommand{\const}{\mathrm{const}}
\newcommand{\id}{\mathrm{id}}
\newcommand{\R}{\mathbb{R}}
\newcommand{\A}{\mathcal{A}}
\newcommand{\B}{\mathcal{B}}

\begin{document}

\title{Attainable set for rank $3$ step $2$ free Carnot group\\ with positive controls
\footnote{The work is supported by the Russian Science Foundation under grant 22-21-00877 (https://rscf.ru/en/project/22-21-00877/) and performed in Ailamazyan Program Systems Institute of Russian Academy of Sciences.}
}

\author{A.\,V.~Podobryaev\\
A.\,K.~Ailamazyan Program Systems Institute of RAS\\
\texttt{alex@alex.botik.ru}}

\date{}

\maketitle

\begin{abstract}
We find the attainable set for a control system on the free Carnot group of rank $3$ and step $2$ with positive controls.
This kind of control systems is connected with the theory of free Lie semigroups;
with some estimates for probabilities of inequalities for independent random variables;
with the nilpotent approximation of robotic control systems and
with contour recovering without cusps in image processing.
We investigate the boundary of the attainable set with the help of the Pontryagin maximum principle for the time-optimal control problem.
We study extremal trajectories that correspond to bang-bang, singular and mixed controls.
We obtain upper bounds for the number of switchings for optimal controls.
This implies a parametrization of the boundary faces of the attainable set.


\end{abstract}

\section{\label{sec-intro}Introduction}
Consider the vector space $\R^r \times \R^{\frac{r(r-1)}{2}}$ of pairs of a vector $x$ and a skew-symmetric matrix $y$.
This vector space equipped with the multiplication rule
\begin{equation}
\label{eq-multiplication}
(x, y) \cdot (x', y') = (x + x', y + y' + (x x'^T - x' x^T))
\end{equation}
is called \emph{a free Carnot group $G$ of rank $r$ and step $2$}.
The following control system
\begin{equation}
\label{eq-controlsystem}
\begin{array}{ll}
\dot{x}_i = u_i, & i=1,\dots,r, \\
\dot{y}_{ij} = x_iu_j - x_ju_i, & i < j \\
\end{array}
\end{equation}
is invariant under the left action of the group $G$.
Control systems on Carnot groups are some kind of cornerstones in geometric control theory~\cite{agrachev-sachkov}
due to existence of a nilpotent approximation for general control systems~\cite{agrachev-sarychev}.
In particular, such control systems appear in several robotic systems~\cite{ardentov-mashtakov} and
in some models for contour reconstruction without cusps in image processing~\cite{mashtakov}.
We consider system~\eqref{eq-controlsystem} with controls $u_1,\dots,u_r \geqslant 0$.
Our goal is a description of the attainable set $\A$ from the point $(0,0)$ for this control system.

For any set of positive numbers $(c_1,\dots,c_r)$ consider \emph{a dilation} that is a map $D_{c_1,\dots,c_r} : G \rightarrow G$ such that
\begin{equation}
\label{eq-dilation}
D_{c_1,\dots,c_r}(x_i, y_{ij}) = (c_ix_i, c_ic_jy_{ij}), \qquad i=1,\dots,r, \qquad i<j.
\end{equation}
Note that dilations preserve system~\eqref{eq-controlsystem}.

Denote by $\A_1 = \{(x, y) \in \A \, | \, x = (1,\dots,1)\}$ the section of the attainable set.
Since $\A$ coincides with the closure of the orbit of $\A_1$ with respect to dilations~\eqref{eq-dilation},
then it is sufficient to describe the section $\A_1$.

H.~Abels and \`{E}.\,B.~Vinberg~\cite{abels-vinberg} suggested a probability interpretation of the section $\A_1$.
It turns out that the coordinates of points of this set are $y_{ij} = 2p_{ij} - 1$ for $i < j$, where
$p_{ij} = P(\xi_i < \xi_j)$ for independent random variables $\xi_1, \dots, \xi_r$ such that
$P(\xi_1 = \dots = \xi_r) = 0$.
So, the set $\B = {{1}\over{2}}((1,\dots,1)^T+\A_1)$ has a probability interpretation.

For $e^{t_1X_{i_1}}e^{t_2X_{i_2}}\dots e^{t_nX_{i_n}} \in \A_1$
from~\eqref{eq-multiplication} we obtain
\begin{equation}
\label{eq-b-coord}
\sum\limits_{l \, | \, i_l = i}{t_i} = 1, \qquad p_{ij} = \sum\limits_{l < m \, | \, i_l = i, i_m = j}{t_lt_m}.
\end{equation}

In paper~\cite{abels-vinberg} the attainable set $\A$ was regarded as \emph{a two-step nilpotent Lie semigroup of rank $r$}
and the set $\B$ (that is equivalent to the section $\A_1$) was described for $r=3$ with the help of an algebraic method.
The result implies some not obvious bounds for probabilities $P(\xi_i < \xi_j)$.
But it seems that the method is unapplicable in more complex cases $r > 3$.

In this paper we propose a constructive and algorithmic approach for $r=3$.
The conjecture is that this method will be useful also for $r > 3$.

The main idea is to consider the time-optimal problem on the group $G$ and
then to describe extremal trajectories that come to the boundary of the attainable set.
They are optimal trajectories for a reduced problem.
The Pontryagin maximum principle implies necessary conditions of optimality.
This allows us to find the extremal trajectories.
Then second order optimality conditions distinguish optimal trajectories.
More precisely, we get the number of control switchings for different kinds of optimal trajectories (bang-bang, singular and mixed).
This implies a parametrization of the boundary faces of the section $\A_1$ of the attainable set.
Note that endpoints of bang-bang trajectories with periodical control sweep non-trivial faces of the boundary of the attainable set.
We give the answer in terms of the set $\B$ and the coordinates $p = p_{12}$, $q = p_{23}$, $r = p_{31}$ on this set, see~\eqref{eq-b-coord}.

The paper has the following structure. In Section~\ref{sec-time-optimal} we introduce a time-optimal problem and necessary conditions for extremal trajectories coming to the boundary of the attainable set. Then we describe types of extremal trajectories in Section~\ref{sec-extremaltrajectories}. We find an upper bound for the number of control switchings for bang-bang trajectories in Section~\ref{sec-bang-bang}. We study singular and mixed trajectories in Sections~\ref{sec-singular} and \ref{sec-mixed}, respectively.
Finally, we summarize obtained results in the main theorem in Section~\ref{sec-conclusion}.

The author would like to thank prof.~A.\,M.~Tsirlin for useful discussions
and E.\,A.~Po\-do\-brya\-eva for the figures.

\section{\label{sec-time-optimal}Time-optimal problem}

In this section we formulate a time-optimal control problem, apply the Pontryagin maximum principle and derive the adjoint subsystem of ODEs that describes extremal trajectories. Here we consider the general case of arbitrary $r \geqslant 2$.

Consider a time-optimal problem
\begin{equation}
\label{eq-time-optimal}
\begin{array}{ll}
(x, y) \in \Lip{([0,T], G)}, & \\
x(0) = 0, & x(T) = x^1, \\
y(0) = 0, & y(T) = y^1, \\
\dot{x}_i = u_i, & i=1,\dots,r, \\
\dot{y}_{ij} = x_iu_j - x_ju_i, & i < j, \\
(u_1,\dots,u_r) \in L^{\infty}([0,T], U), & T \rightarrow \mathrm{min},
\end{array}
\end{equation}
where the set of controls is a simplex
\begin{equation*}
U = \{(u_1,\dots,u_r) \, | \, u_1 + \dots + u_r = 1, \, u_1, \dots, u_r \geqslant 0 \}.
\end{equation*}

\begin{proposition}
\label{prop-optimaltillinfty}
Every admissible trajectory of control system~$\eqref{eq-time-optimal}$ is optimal.
\end{proposition}

\begin{proof}
Using an argument from~\cite{ardentov-ledonne-sachkov1}, we obtain that the distance from a point of an admissible trajectory to the plane $x_1 + \dots + x_r = 1$ is proportional to the time.
\end{proof}

\begin{proposition}
\label{prop-equivalent}
An optimal trajectory of problem~$\eqref{eq-time-optimal}$ comes to the boundary of the attainable set iff
it is time-optimal for the problem~$\eqref{eq-time-optimal}$ factorized by the subspace $\R(1,\dots,1)^T$
(so called reduced problem).
\end{proposition}

\begin{proof}
The original problem is an extended problem of the reduced one, see for example~\cite{agrachev-sachkov}.
So, optimal trajectories of the reduced problem lift to the trajectories coming to the boundary of the attainable set of the original one.
\end{proof}

For $i=1,\dots, r$ let $X_i$ be a basis vector field corresponding to coordinate $x_i$.
Denote by $h_i = \langle X_i, \, \cdot \, \rangle$ linear on the fibers of the cotangent bundle $T^*G$ functions.
Introduce a family of functions on $T^*G$ depending on a parameter $u = (u_1,\dots,u_r)^T$:
\begin{equation}
\label{eq-hamiltonian}
H_u(\lambda) = u_1h_1(\lambda) + \dots + u_rh_r(\lambda), \qquad \lambda \in T^*G.
\end{equation}

Necessary conditions of optimality follow from the Pontryagin maximum principle~\cite{pontryagin,agrachev-sachkov}.
These conditions are the same for extremal trajectories of problem~\eqref{eq-time-optimal} coming to the boundary of the set $\A$
and for optimal trajectories of the reduced control system.

\begin{theorem}
\label{th-pmp}
If $(\tilde{x},\tilde{y})$ and $\tilde{u}$ are an optimal process for problem~$\eqref{eq-time-optimal}$, then
there exist a curve $\lambda \in \Lip{([0,T], T^*G)}$, $\pi(\lambda(t)) = (\tilde{x}(t), \tilde{y}(t))$
and a number $\nu \in \{0, 1\}$ such that for a.e. $t \in [0, T]$ we have
\begin{equation*}
\begin{array}{ll}
\dot{\lambda}(t) = \vec{H}_{\tilde{u}(t)}(\lambda(t)), & \lambda(t) \neq 0,\\
H_{\tilde{u}(t)}(\lambda(t)) = \max\limits_{u \in U}{H_u(\lambda(t))}, & H_{\tilde{u}(t)}(\lambda(t)) \equiv \nu,\\
\end{array}
\end{equation*}
where $\pi : T^*G \rightarrow G$ is the natural projection and $\vec{H}$ is the Hamiltonian vector field corresponding to a Hamiltonian $H$.
\end{theorem}

Recall that the curve $\lambda$ is called \emph{an extremal} and
its projection $\pi(\lambda)$ is called \emph{an extremal trajectory}.

Let $Y_{ij} = [X_i, X_j]$ and $h_{ij} = \langle Y_{ij}, \, \cdot \, \rangle$.
Then the adjoint subsystem of the Hamiltonian system of the Pontryagin maximum principle reads as
\begin{equation}
\label{eq-vertpart}
\dot{h}(t) = R \tilde{u}(t), \qquad \dot{R} = 0,
\end{equation}
where $h = (h_1,\dots,h_r)^T$ and $R = (h_{ij})$ is a skew-symmetric matrix.

It is easy to see that if $\nu = 0$, then $u = 0$ and the corresponding extremal trajectory is constant $(0,0)$.
So, we assume that $\nu = 1$.
It follows from the maximum condition that the adjoint trajectories $h(\,\cdot\,)$ for extremals lie on the boundary of the quadrant
\begin{equation*}
Q = \{(h_1,\dots,h_r)^T \in (\R^r)^* \, | \, h_1,\dots,h_r \leqslant 1\}
\end{equation*}
that is a dual set to the set of controls $U$.

\section{\label{sec-extremaltrajectories}Types of extremals}

Here and below we fix $r=3$.
Introduce some notation.
For $i=1,2,3$ let $F_i$ be the face of the quadrant $Q$ defined by the condition $h_i = 1$.
Next, for $i,j=1,2,3$ denote by $E_{ij} = F_i \cap F_j$ the edges of $Q$.
Finally, let $E = E_{12} \cup E_{23} \cup E_{31}$ be the union of the edges and $V = (1,1,1)^T$ be the vertex.

Let $h:[0,T] \rightarrow (\R^3)^*$ be a trajectory of the adjoint subsystem.
Note that if $h(t)$ is inside some face of the quadrant $Q$, then the corresponding control is at the vertex of the triangle $U$.
If $h(t)$ is on some edge of $Q$, then the corresponding control lies on the edge of the triangle $U$.
If $h(t) = V$, then the control may be anywhere at the triangle $U$.

\begin{proposition}
\label{prop-types-of-extremals}
For problem~$\eqref{eq-time-optimal}$ there are three types of extremals $\lambda : [0, T] \rightarrow T^*G$:\\
$(1)$ bang-bang, i.e., such that
\begin{equation*}
\card{\{t \in [0, T] \, | \, h(t) \in E\}} < \infty;
\end{equation*}
$(2)$ singular, i.e., $h(t) \in E_i$ for some $i=1,2,3$ or $h(t) \in V$;\\
$(3)$ mixed, i.e., a concatenation of a finite number of singular and bang-bang arcs,
where $h:[0,T] \rightarrow (\R^3)^*$ is the corresponding trajectory of the adjoint subsystem.
\end{proposition}

\begin{proof}
The function
\begin{equation}
\label{eq-Casimir}
C = h_1h_{23} + h_2h_{31} + h_3h_{12}
\end{equation}
is a Casimir function.
In particular, it is an integral of system~\eqref{eq-vertpart}.
Hence, the trajectories $h(t)$ lie in the intersections of the boundary of the quadrant $Q$ with
two-dimensional affine subspaces $C = \const$.
Possible situations are presented in Fig.~\ref{fig-bang-bang} for bang-bang trajectories and the singular trajectory at the vertex $V$ and Fig.~\ref{fig-singular} where non-trivial singular and mixed arcs appear.
\end{proof}

\begin{figure*}[h]
\centering
\subfloat[$1$ or $2$ switchings]{\includegraphics[width=2.5in]{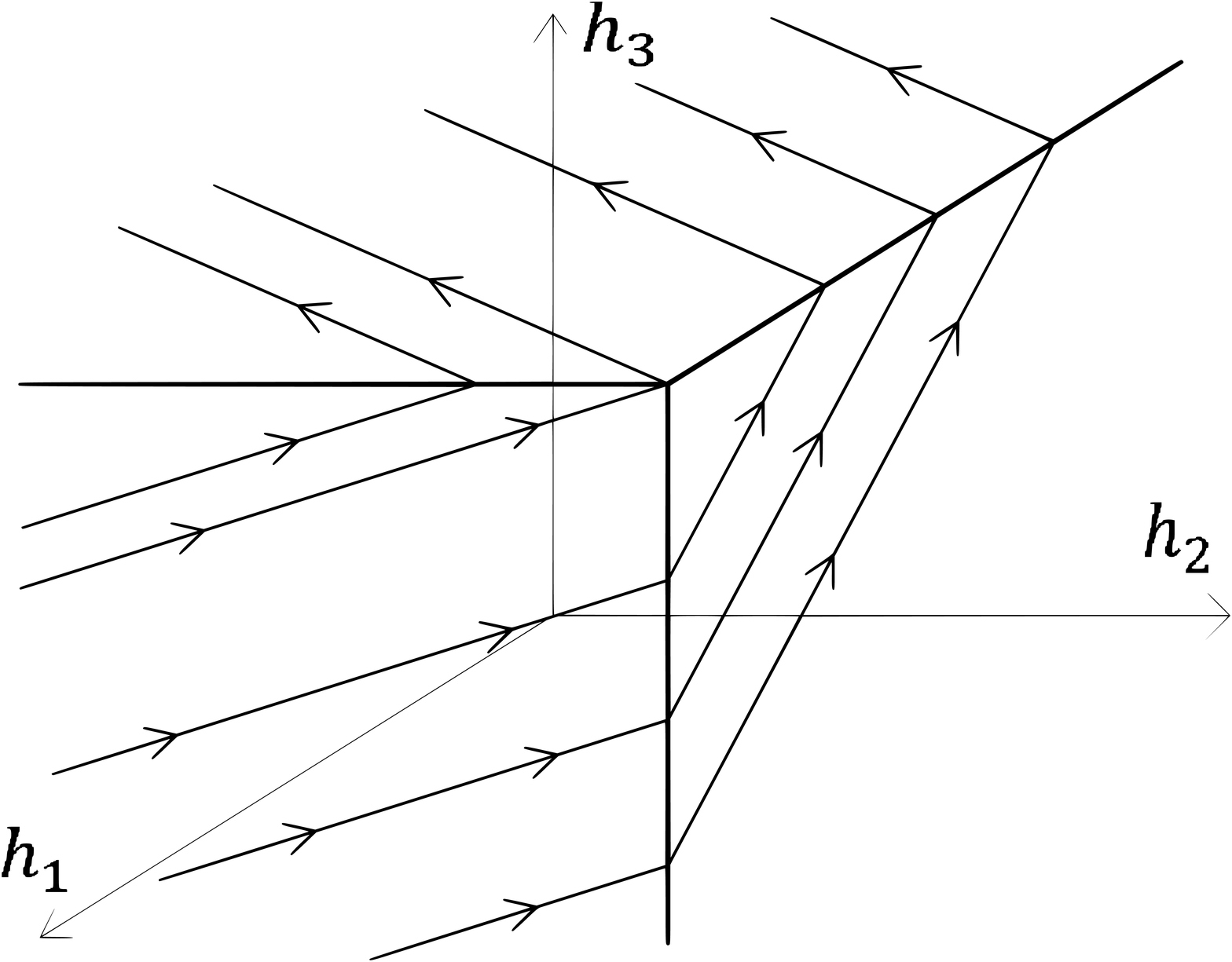}
\label{fig-bang-bang-1-2}}
\hfil
\subfloat[periodical control]{\includegraphics[width=2.5in]{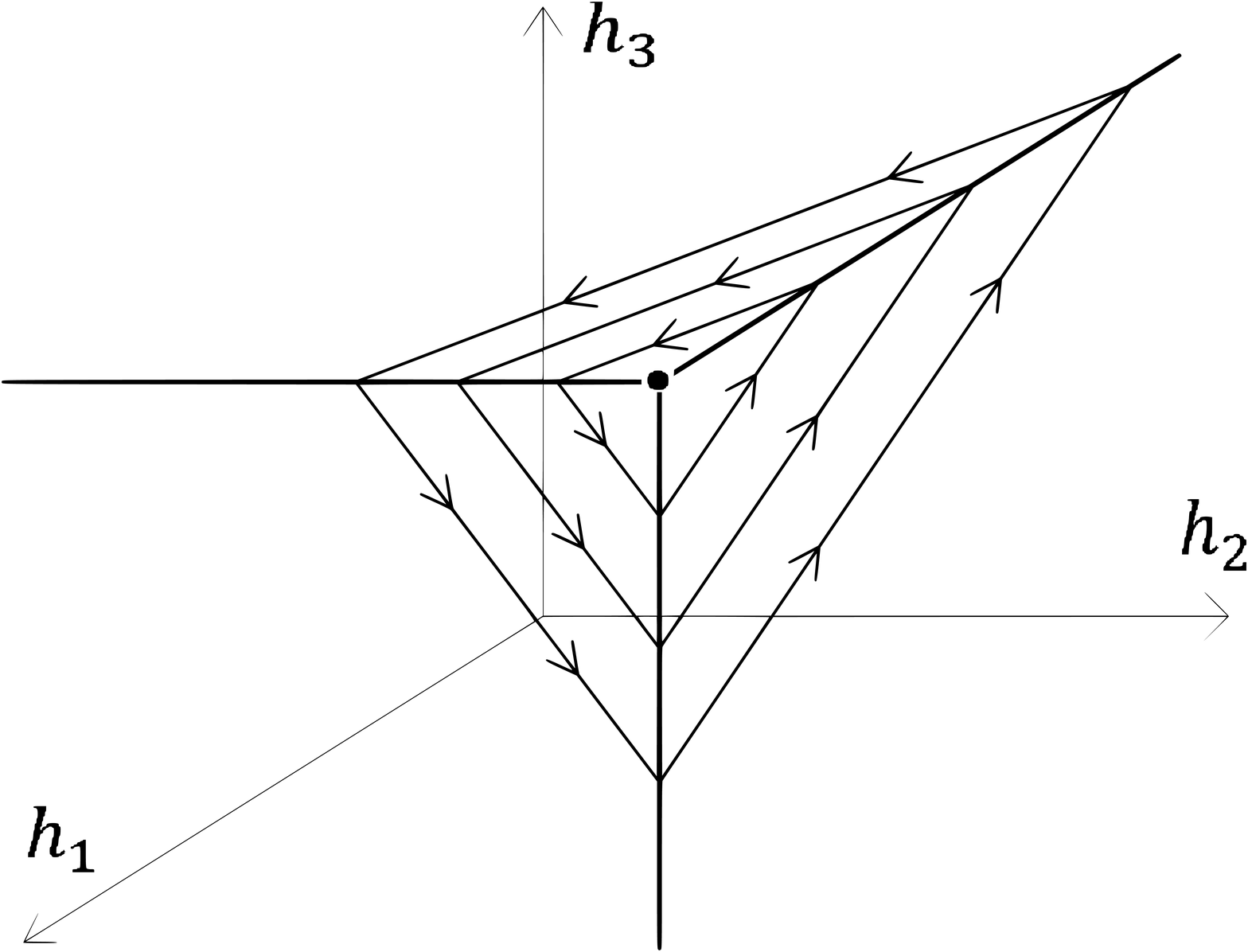}
\label{fig-bang-bang-3}}
\caption{Adjoint subsystems for bang-bang trajectories and one singular trajectory.}
\label{fig-bang-bang}
\end{figure*}

\begin{figure*}[h]
\centering
\subfloat[concatenation of singular and bang arcs]{\includegraphics[width=2.5in]{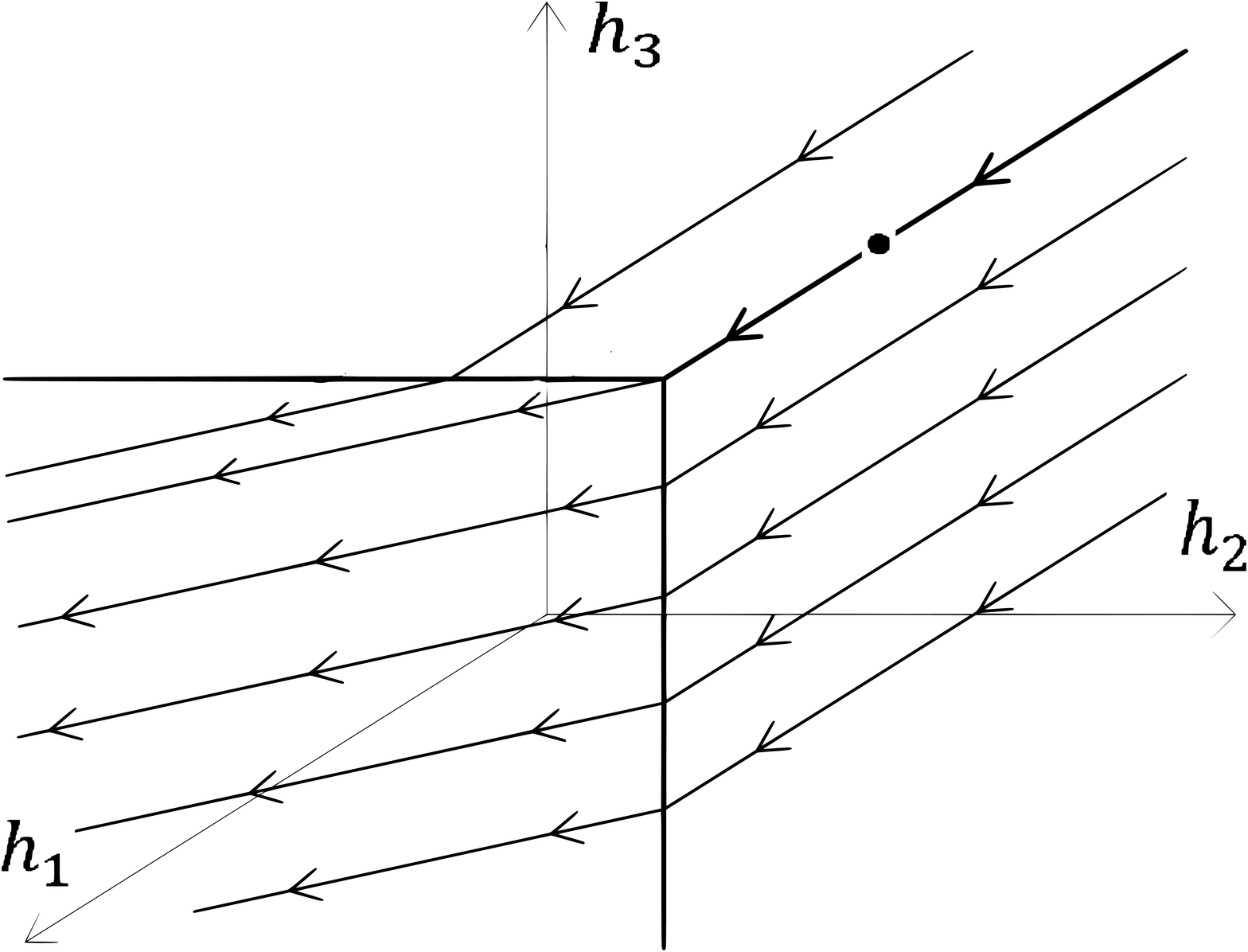}
\label{fig-singular-bang-bang}}
\hfil
\subfloat[concatenation of two singular arcs]{\includegraphics[width=2.5in]{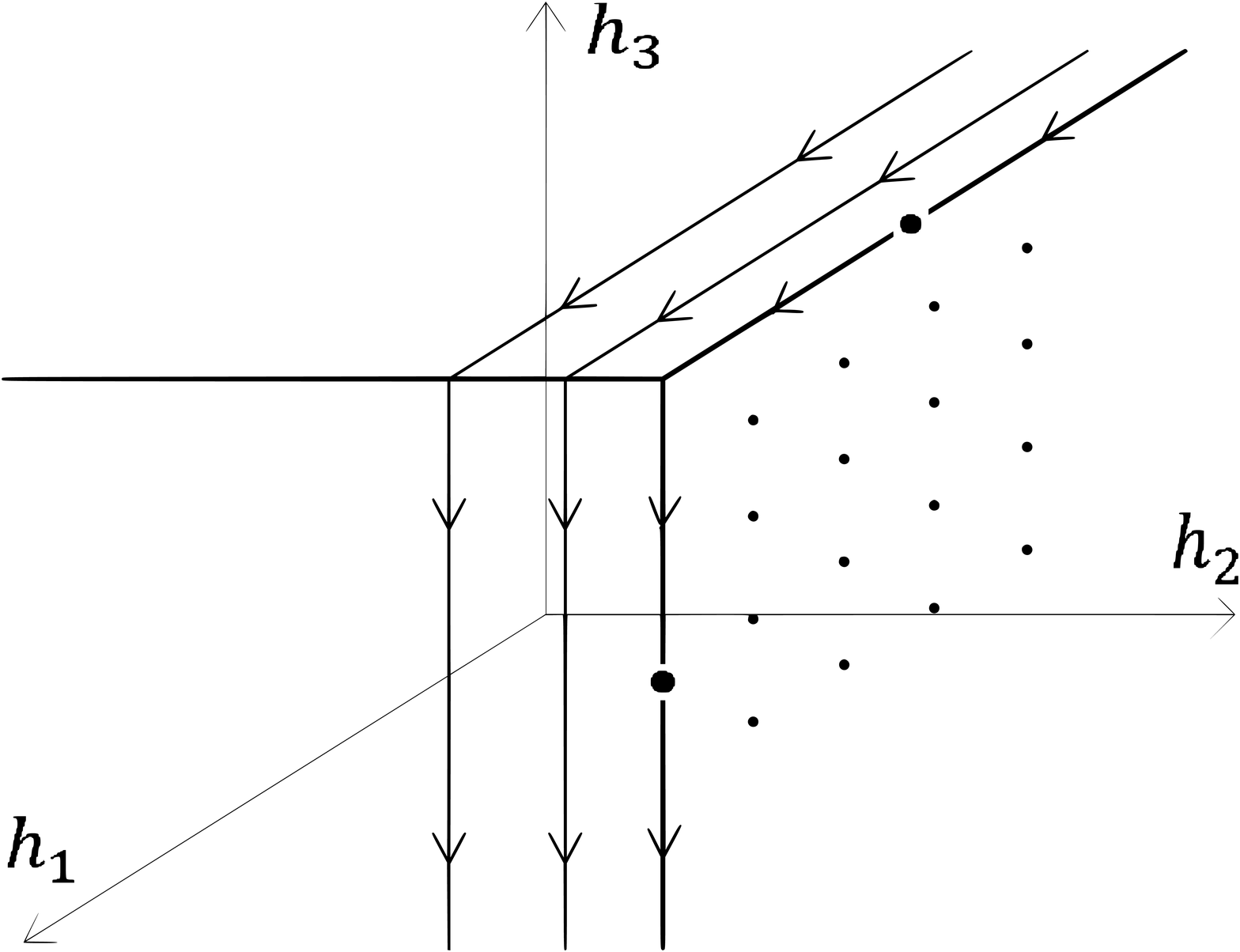}
\label{fig-singular-singular}}
\caption{Adjoint subsystems with singular trajectories.}
\label{fig-singular}
\end{figure*}

\section{\label{sec-bang-bang}Bang-bang trajectories}

In this section we get an upper bound for the number of switchings on optimal bang-bang curves using the following
theorem by A.\,A.~Agrachev and R.\,V.~Gamkrelidze~\cite{agrachev-gamkrelidze}.
Earlier this method was successfully applied for analysis of bang-bang extremals on the Cartan group~\cite{ardentov-ledonne-sachkov2}.

\begin{theorem}
\label{th-ag}
Assume that $q(\,\cdot\,)$ is an extremal curve, $u(\,\cdot\,)$ is an extremal control of problem~$\eqref{eq-time-optimal}$
and $\lambda(\,\cdot\,)$ is the corresponding extremal that is unique upto a multiplication by a positive scalar.
Let $u(t) = u^i$ for $t \in (t_i, t_{i+1})$, where $0 = t_0 < t_1 < \dots < t_{k+1} = T$, and
let $V_i$ be the velocity corresponding to the control $u^i$.
Define recursively the following operators:
\begin{equation*}
P_0 = P_1 = \id, \qquad P_i = P_{i-1} \circ e^{(t_i - t_{i-1})\ad{V_{i-1}}}, \qquad i=2,\dots,k.
\end{equation*}
Define the vector fields $Z_i = P_i V_i$.
If the quadratic form
\begin{equation*}
G(\alpha) = \sum\limits_{0 \leqslant i < j \leqslant k}{\alpha_i\alpha_j \langle \lambda(t_1), [Z_i, Z_j](q(t_1)) \rangle}
\end{equation*}
is not negative semidefinite on the space
$$
W=\{(\alpha_0, \dots, \alpha_k) \in \R^{k+1} \bigm|
\sum\limits_{i=0}^{k}{\alpha_i} = 0, \, \sum\limits_{i=0}^{k}{\alpha_iZ_i(q(t_1))} = 0 \},
$$
then the trajectory $q(\,\cdot\,)$ is not optimal.
\end{theorem}

\begin{proposition}
\label{prop-bang-bang-more4}
If an extremal curve has more than 4 switchings, then it is not optimal.
\end{proposition}

\begin{proof}
Let us apply Theorem~\ref{th-ag} to an extremal curve with $k = 5$ switchings.
For definiteness let
\begin{equation*}
\begin{array}{lll}
V_0 = X_1, & V_1 = X_2, & V_2 = X_3, \\
V_3 = X_1, & V_4 = X_2, & V_5 = X_3 \\
\end{array}
\end{equation*}
and $\tau_i = t_i-t_{i-1}$ for $i=1,\dots,5$.
We have
\begin{equation*}
\begin{array}{c}
Z_0 = X_1, \qquad Z_1 = X_2, \qquad Z_2 = X_3 + \tau_2 Y_{23},\\
Z_3 = X_1 - \tau_2 Y_{12} - \tau_3 Y_{13}, \qquad Z_4 = X_2 - \tau_3 Y_{23} + \tau_4 Y_{12},\\
Z_5 = X_3 + (\tau_2 + \tau_5) Y_{23} + \tau_4 Y_{13}. \\
\end{array}
\end{equation*}
Note that $\tau_5 = \tau_2$, since this is the time of passing the same edge of a triangle by the extremal trajectory,
see Fig.~\ref{fig-bang-bang-3}.
Hence, from the system of equations for the subspace $W$ we get
\begin{equation*}
\begin{array}{lll}
\alpha_0 = -\frac{\tau_4}{\tau_3}\alpha_5, & \alpha_1 = -\frac{\tau_2}{\tau_3}\alpha_5, & \alpha_2 = -\alpha_5,\\
\alpha_3 = \frac{\tau_4}{\tau_3}\alpha_5, & \alpha_4 = \frac{\tau_2}{\tau_3}\alpha_5, & \alpha_5 \in \R.\\
\end{array}
\end{equation*}
Next, for quadratic form $G$ we obtain
\begin{equation*}
G(\alpha) = \frac{2\tau_2\tau_4}{\tau_3}\left(\frac{h_{12}}{\tau_3} + \frac{h_{23}}{\tau_4} - \frac{h_{31}}{\tau_2}\right)\alpha_5^2.
\end{equation*}
Note that since we have a triangle as a section of $Q$ and then the coefficients of $C$ are positive $h_{12}, h_{23}, h_{31} > 0$.
Using~\eqref{eq-Casimir} it is not difficult to see that
\begin{equation*}
\tau_2 = \frac{K}{h_{12}h_{23}}, \qquad \tau_3 = \frac{K}{h_{23}h_{31}}, \qquad \tau_4 = \frac{K}{h_{31}h_{12}},
\end{equation*}
where $K = h_{12}+h_{23}+h_{31} - C > 0$.
Therefore, the sign of $G(\alpha)$ coincide with the sign of
$h_{12}h_{23}h_{31} + h_{23}h_{31}h_{12} - h_{31}h_{12}h_{23} = 2 h_{12}h_{23}h_{31} > 0$.
Thus, the quadratic form $G$ is positive definite on the subspace $W$.
So, the extremal trajectory under consideration can not be optimal.
\end{proof}

Now we describe a part of the boundary $\partial \A_1$ (or equivalently $\partial \B$) corresponding to the endpoints of bang-bang extremal curves with the number of switchings less than $4$.

\begin{proposition}
\label{prop-ban-bang-less4}
Bang-bang trajectories with not more than $3$ switchings correspond to $\partial \B$.
\end{proposition}

\begin{proof}
Obviously $\B \subset [0,1]^3$.
Using~\eqref{eq-b-coord} we conclude, that
trajectories without switchings or with one switching correspond to the vertices $A_1 = (1,0,0)$, $B_2 = (0,1,0)$, $C_1 = (0,0,1)$ of the cube $[0,1]^3$, see Fig.~\ref{fig-cube}.
Next, trajectories with two switchings match to the vertices $A_2 = (1,0,1)$, $C_2 = (0,1,1)$, $D_1 = (1,1,0)$.
Finally, trajectories with three switchings respond to the edges of the triangles $A_1B_2C_1$ and $A_2C_2D_1$.
These edges lie on the boundary of the cube.
This implies that these edges are in $\partial \B$.
\end{proof}

\begin{figure}[h]
\centerline{\includegraphics[width=0.35\linewidth]{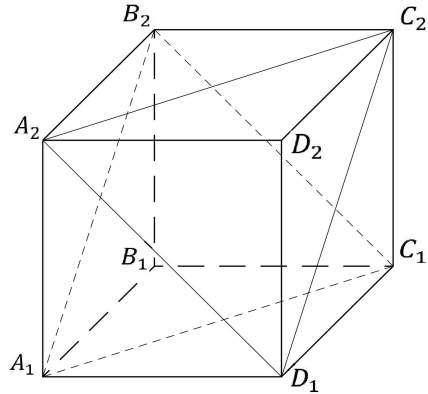}}
\caption{The cube in the $(p,q,r)$-space containing the set $\B$.}
\label{fig-cube}
\end{figure}

\begin{proposition}
\label{prop-bang-bang-4}
Extremal trajectories with $4$ switchings correspond to quadratic surfaces in $\B$ of the type
$$
p + qr = 1 \qquad \text{or} \qquad (1-p) + (1-q)(1-r) = 1
$$
for any cyclic permutations of the variables
$$
p = P(\xi_1 < \xi_2), \qquad q = P(\xi_2 < \xi_3), \qquad r = P(\xi_3 < \xi_1).
$$
\end{proposition}

\begin{proof}
This follows directly from~\eqref{eq-b-coord}.
The sequence of controls defines the type of the equation.
Even permutations of controls correspond to the first type of equations and odd permutations correspond to the second one.
\end{proof}

\section{\label{sec-singular}Singular trajectories}

The next proposition allows us to consider singular trajectories with a control on the boundary $\partial U$.

\begin{proposition}
\label{prop-singular-boundary}
If a control for an extremal trajectory lies in the interior of the control set $U$, then this trajectory lies in the interior of the attainable set $\A$.
\end{proposition}

\begin{proof}
Assume by contradiction that this trajectory comes to the boundary $\partial \A$.
It follows that this trajectory is optimal for the reduced problem.
Project the corresponding control from the center of the set $U$ to its boundary and
use this projection as a new control.
Thus, this new control allows to go faster along the same extremal trajectory of the reduced problem.
\end{proof}

Therefore, the study of singular trajectories corresponding to the vertex $V$ is reduced to the study of bang-bang trajectories,
singular trajectories corresponding to an edge of the quadrant $Q$ and mixed trajectories.

Let a singular trajectory corresponds to the edge $E_{ij}$.
Consider a free two-step Carnot subgroup $H_{ij}$ such that its Lie algebra is generated by the elements $X_i$ and $X_j$.
(The subgroup $H_{ij}$ is isomorphic to \emph{the Heisenberg group}.)
Denote the attainable set for the subgroup $H_{ij}$ by $\A(H_{ij})$.

\begin{proposition}
\label{prop-singular-edge}
It is sufficient to consider singular curves corresponding to the edge $E_{ij}$ with two switchings
to attain the set $\partial \A(H_{ij})$.
Three switchings are sufficient to attain any point of the set $\A(H_{ij})$.
\end{proposition}

\begin{proof}
Since the factor of the set $\A(H_{ij})$ by dilations is just an interval $\A_1(H_{ij}) = [-1, 1]$,
then this statement could be checked directly.
\end{proof}

\section{\label{sec-mixed}Mixed trajectories}

\begin{proposition}
\label{prop-mixed-singular-bang-bang}
The endpoints of mixed trajectories that are concatenation of a singular arc corresponding to an edge of the quadrant $Q$ and a bang arc
\emph{(}see Fig.~\emph{\ref{fig-singular-bang-bang})} match to the edges of the set $\B$.
\end{proposition}

\begin{proof}
It can be checked directly using~\eqref{eq-b-coord} that the ends of these trajectories correspond to the edges
$A_1A_2$, $A_2B_2$, $B_2C_2$, $C_2C_1$, $C_1D_1$, $D_1A_1$ of the cube.
Thus, these endpoints correspond to $\partial \B$.
\end{proof}

\begin{proposition}
\label{prop-mixed-singular-singular}
The endpoints of mixed trajectories that are concatenation of two singular arcs corresponding to two edges of the quadrant $Q$
\emph{(}see Fig.~\emph{\ref{fig-singular-singular})} match to
the triangular faces $A_1A_2B_2$, $B_2C_2C_1$ and $C_1D_1A_1$ of the set $\B$.
\end{proposition}

\begin{proof}
Consider two edges of the quadrant $Q$, say $E_{12}$ and $E_{23}$.
The attainable set for such mixed trajectories is $\A(H_{12}) \cdot \A(H_{23}) = \partial \A(H_{12}) \cdot \A(H_{23})$.
The last equality follows from Proposition~\ref{prop-singular-edge}.
So, it is sufficient to make one switching on the first singular arc and two switchings on the second one.
The ends of these trajectories correspond to the triangles $A_1A_2B_2$, $B_2C_2C_1$ and $C_1D_1A_1$.
Since these triangles lie on the boundary of the cube, then they are parts of the boundary $\partial \B$.
\end{proof}

\section{\label{sec-conclusion}Conclusion}

We considered all types of extremal trajectories for the time-optimal problem~\eqref{eq-time-optimal} and now we can summarize the result.

\begin{theorem}
\label{th-main}
$(1)$ The attainable set $\A$ for a free two-step Carnot group of rank $3$
can be obtained applying all dilations~$\eqref{eq-dilation}$ to the section $\A_1 = 2\B - (1,\dots,1)^T$ and
taking the closure.\\
$(2)$ The set $\B$ is a curved polyhedron carved from the cube $[0,1]^3$ with $(p,q,r)$-coordinates by the surfaces
\begin{equation*}
p + qr = 1, \qquad (1-p) + (1-q)(1-r) = 1
\end{equation*}
and surfaces obtained from these surfaces by any permutations of the variables $p, q, r$.\\
$(3)$ The vertices, edges and faces of $\B$ correspond to the endpoints of the extremal trajectories of problem~$\eqref{eq-time-optimal}$ with not more than $2$, $3$ and $4$ switchings, respectively.
\end{theorem}

\end{document}